\documentclass[reqno, 12pt, reqno]{amsart}

\usepackage{amsfonts, amsthm, amsmath, amssymb}
\usepackage{hyperref}
\hypersetup{colorlinks=false}

\usepackage[margin=1.5in]{geometry}

\RequirePackage{mathrsfs} \let\mathcal\mathscr
 
\numberwithin{equation}{section}

\newtheorem{theorem}{Theorem}[section]
\newtheorem{lemma}[theorem]{Lemma}

\newtheorem{cor}[theorem]{Corollary}

\theoremstyle{definition}

\renewcommand{\phi}{\varphi}
\renewcommand{\rho}{\varrho}

\newcommand{\card}{\#}
\newcommand{\0}{\mathbf{0}}
\newcommand{\PP}{\mathbb{P}}
\renewcommand{\AA}{\mathbb{A}}
\newcommand{\A}{\mathbf{A}}
\newcommand{\FF}{\mathbb{F}}
\newcommand{\ZZ}{\mathbb{Z}}
\newcommand{\Z}{\mathbb{Z}}
\newcommand{\N}{\mathbb{N}}

\newcommand{\NN}{\mathbb{N}}
\newcommand{\QQ}{\mathbb{Q}}
\newcommand{\RR}{\mathbb{R}}
\newcommand{\R}{\mathbb{R}}

\newcommand{\cN}{\mathcal{N}}

\newcommand{\val}{{\rm val}}

\renewcommand{\leq}{\leqslant}
\renewcommand{\le}{\leqslant}
\renewcommand{\geq}{\geqslant}
\renewcommand{\ge}{\geqslant}

\newcommand{\x}{\mathbf{x}}
\newcommand{\y}{\mathbf{y}}

\renewcommand{\a}{\mathbf{a}}

\newcommand{\ve}{\varepsilon}
\newcommand{\ep}{\varepsilon}

\newcommand{\e}{\mathbf{e}}

\DeclareMathOperator{\rank}{rank}

\DeclareMathOperator{\meas}{meas}

\DeclareMathOperator{\Spec}{Spec}

\newcommand{\beql}[1]{\begin{equation}\label{#1}}
\newcommand{\eeq}{\end{equation}}
\newcommand{\sL}{\mathsf{\Lambda}}
\renewcommand{\mod}[1]{\; ( \text{mod} \; #1)}

\renewcommand{\b}[1]{\mathbf{#1}}

\begin{document}

\date{\today}

\title{The geometric sieve for quadrics} 
\author{T.D.\ Browning}

\address{IST Austria\\
Am Campus 1\\
3400 Klosterneuburg\\
Austria}
\email{tdb@ist.ac.at}

\author{D.R.\ Heath-Brown}
\address{Mathematical Institute\\
Radcliffe Observatory Quarter\\ Woodstock Road\\ Oxford\\ OX2 6GG\\ United Kingdom}
\email{rhb@maths.ox.ac.uk}


\begin{abstract}
We develop a version of Ekedahl's geometric sieve for   integral  quadratic forms of  rank at least five.  As one ranges over the zeros of such quadratic forms, we use the sieve to compute the density of coprime values of polynomials, and furthermore, to address a question about  local solubility in  families of varieties 
parameterised by the zeros.
\end{abstract}

\date{\today}

\maketitle

\thispagestyle{empty}

\setcounter{tocdepth}{1}
\tableofcontents

\section{Introduction}

The geometric sieve  originates in
pioneering work of Ekedahl \cite{ekedahl}.
It is usually taken to mean that 
for any 
codimension $2$ subvariety $Z\subset \AA_\ZZ^n$ that is defined over
$\ZZ$,  the asymptotic proportion of lattice points in a homogeneously
expanding region in $\RR^n$ that reduce modulo $p$ 
 to an $\FF_p$-point of $Z$, for some  prime $p>M$, approaches zero as
 $M\to \infty$.   
Bhargava   \cite[Thm.~3.3]{sieve} has  established a precise
quantitative version of Ekedahl's result.  This basic fact has yielded an impressive  array of 
applications in arithmetic statistics.

The earliest application of the geometric sieve concerned 
 relatively prime polynomials $f,g\in \ZZ[X_1,\dots,X_n]$. 
It was shown  by 
Ekedahl \cite{ekedahl} that 
the density of 
$n$-tuples of positive integers for which the values of $f$ and $g$ are coprime
is equal to $\prod_{p}(1-c_p p^{-n})$, where 
$$
c_p=\#\{\x\in (\ZZ/p\ZZ)^n: f(\x)\equiv g(\x)\equiv 0 \mod{p}\}.
$$
This result 
has since been generalised and extended to function fields of positive
characteristic by Poonen  
  \cite[Thm.~3.1]{poonen}.

Next, when degree $d$  hypersurfaces $X\subset \PP^m$ with rational coefficients 
  are  ordered by height, a positive proportion are everywhere locally
  soluble,  provided that $(d,m)\neq (2,2)$. This application of the geometric sieve 
  is due to Poonen and
  Voloch \cite[Thm.~3.6]{MR2029869}, but has been  
extended to more general families of varieties $Y\to \PP^n$ over  arbitrary number fields
by Bright, Browning and Loughran 
\cite[Thm.~1.3]{wa}. 

The geometric sieve has also proved instrumental in questions about square-free values of polynomials.  For example, 
using the geometric sieve, 
Bhargava, Shankar and Wang \cite{poly} 
have recently determined the precise  density of  
monic integer polynomials of fixed  degree that   have square-free
discriminant.

Very recently Cremona and Sadek \cite{CS} have used the geometric 
sieve to investigate the proportion of integral Weierestrass equations
of elliptic curves (when ordered by height) which are, for example,
globally minimal. They estalish a form of the sieve which applies to
boxes of unequal sides, somewhat in the spirit of Lemma \ref{t:ekedahl}
below, though less general.

The primary goal of this paper is to achieve a version of the geometric sieve which works for
codimension $2$ subvarieties of aribtrary smooth projective quadrics of rank at least 5.

 \begin{theorem}\label{MT*}
 Let  $X\subset \PP^m$ be a hypersurface defined over $\QQ$
 by a quadratic form of rank at least $5$. Let
 $Z\subset X$ be a codimension $2$ subvariety defined over $\QQ$, let
 $\mathcal{Z}$ be its scheme-theoretic closure in $\PP_\ZZ^m$,  and
 let  
$Z_p=\mathcal{Z} \otimes_\ZZ \FF_p$, for any prime $p$.
   Then  for any $\ve>0$ 
there exists a  constant 
$c_{\ve,X,Z}>0$  depending  only on $X,Z$ and  $\ve$, such that 
the number of 
$x\in X(\QQ)$ of height $H(x)\leq B$ which specialise to a point in
$Z_p(\FF_p)$,  for some $p>M$, is at most 
$$
c_{\ve,X,Z} B^\ve\left(\frac{B^{m-1}}{M\log M} +B^{m-1-1/m}\right).
$$
 \end{theorem}
 
The height function $H $ in Theorem \ref{MT*}   is the naive exponential
height on $\PP^{m}(\QQ)$. For $X$ as in the theorem, the
Hardy--Littlewood circle method   ensures 
that either $X(\RR)=\emptyset$ or there is
a constant $c_X>0$ such that  
$$
\#\{x\in X(\QQ): H(x)\leq B\}\sim c_X B^{m-1},
$$
as $B\to \infty$. This follows from work of Birch
\cite{birch}, for example.  Theorem \ref{MT*}  therefore implies 
 that it is rare for rational points on 
$X$ to specialise to points on $Z_p(\FF_p)$ for large primes $p$.

 We shall prove Theorem \ref{MT*} in the following more explicit form.
\begin{theorem}\label{MT+}
Let 
$Q(X_0,\dots,X_n)$ be a quadratic form defined over $\ZZ$ with rank
at least $5$, and let
$F_1(X_0,\dots, X_n),\dots,F_r(X_0,\dots, X_n)$
be forms defined over $\ZZ$. Assume that the variety
$Z\subset\mathbb{P}^{n}$ given by
\[Z:\, Q(X_0,\dots,X_n)=F_1(X_0,\dots, X_n)=
\dots=F_r(X_0,\dots, X_n)=0\]
has codimension at least $3$ in $\mathbb{P}^n$.
For $B,M\ge 1$ let  $N(B,M)$  be the number of vectors
$\x\in\ZZ^{n+1}$ such that
\[Q(x_0,\dots,x_n)=0,\]
with  $|\x|\leq B$, and for which 
$F_1(x_0,\dots,x_n),\dots,F_r(x_0,\dots,x_n)$ have a
common prime divisor $p> M$. Then
\[N(B,M)\ll_{\ve, Q,F_1,\dots,F_r} \frac{B^{n-1+\ve}}{M\log M}+B^{n-1-1/n+\ve},\]
for any fixed $\ep>0$.
\end{theorem}

Here we write $|\cdot|$ for the supremum norm
$||\cdot||_{\infty}$ on $\RR^m$ for any $m\in\NN$.
These results could be false when the underlying quadratic form has
rank less than 5. For example, if $n\ge 3$ and
$$
Q(X_0,\dots,X_n)=X_0X_1-X_2X_3,$$
or 
$$
Q(X_0,\dots,X_n)=X_0X_1-X_3^2,
$$
then we may take $Z$ to be the linear space $X_1=X_2=X_3=0$. If
$M$ is in the range
$B^{1/2}<M\le B^{3/4}$, say, then  we may consider points 
$$
(a,0,bp,0,x_4,\dots,x_n)
$$ 
of height at most $B$, where $p$
ranges over primes in the interval 
$M<p\le B$, and 
$\gcd(a,bp)=1$. There will be at least $cB^{n-1}$ such points, for a
suitable absolute constant $c>0$. Moreover each of them lies on $Q=0$,
and each of them reduces to a point of $Z$ modulo the relevant
prime $p$.

\medskip
A result similar in spirit to Theorem \ref{MT*} has been proved
simultaneously by Cao and Huang \cite[Thm.~4.7]{CH}, for affine quadrics defined
by 
$$
Q(X_1,\ldots,X_n)=m,
$$ 
with $m$ a non-zero integer.
Their result is more delicate
than ours, saving only a factor $\sqrt{\log B}$.
\medskip

The case in which the quadric hypersurface has no non-singular rational
point is uninteresting,  but the examples above leave open the
situation in which the quadratic form takes the shape
\[Q(X_0,\dots,X_n)=X_0X_1-(X_2^2-dX_4^2),\]
for some non-square $d\in\ZZ$. This is covered in the following
theorem.
\begin{theorem}\label{MT++}
Let 
$Q(X_0,\dots,X_n)$ be a quadratic form defined over $\ZZ$, equivalent
over $\QQ$ to a non-zero multiple of $X_0X_1-(X_2^2-dX_4^2)$ for
some non-square $d\in\ZZ$. Let
$F_1(X_0,\dots, X_n),\dots,F_r(X_0,\dots, X_n)$
be forms defined over $\ZZ$. Assume that the variety
$Z\subset\mathbb{P}^{n}$ given by
\[Z:\, Q(X_0,\dots,X_n)=F_1(X_0,\dots, X_n)=
\dots=F_r(X_0,\dots, X_n)=0\]
has codimension at least $3$ in $\mathbb{P}^n$.
Then 
\[N(B,M)\ll_{\ve, Q,F_1,\dots,F_r} \frac{B^{n-1+\ve}}{M\log M}+B^{n-3/2+\ve},\]
for any fixed $\ep>0$, 
where $N(B,M)$ is defined in Theorem \ref{MT+} for  $B,M\geq 1$.
\end{theorem}

Our proof of Theorem \ref{MT++}  will be a non-trivial variant of that for Theorem~\ref{MT+}. 
\medskip

It is natural to ask what applications 
are available for   our version of the  geometric sieve for quadrics.
We first demonstrate that the result
of Ekedahl \cite{ekedahl} and Poonen 
  \cite[Thm.~3.1]{poonen} about coprime values of polynomials remains
  true when one restricts to the much thinner set of zeros of a given
  quadratic form.  
For any $\mathcal{S}\subset \ZZ^n$ and any non-singular quadratic form
$Q\in \ZZ[X_1,\dots,X_n]$ we define  
\begin{equation}\label{eq:muQ}
\mu_Q(\mathcal{S})=\lim_{B\to \infty} \frac{\#\{\x\in
  \mathcal{S}\cap[-B,B]^n: Q(\x)=0\}}{\#\{\x\in \ZZ^n\cap[-B,B]^n:
  Q(\x)=0\}}, 
\end{equation}
if the  limit exists.  Given polynomials $f,g\in 
\ZZ[X_1,\dots,X_n]$, let 
$$
\mathcal{R}_{f,g}=\{\x\in \ZZ^n: \gcd(f(\x),g(\x))=1\}.
$$
We shall prove the following result in Section \ref{s:coprime}.

\begin{cor}\label{cor:gcd}
Assume that $Q$ is indefinite and has rank at least $5$.
Let 
$f,g\in 
\ZZ[X_1,\dots,X_n]$ be homogenous, such that the variety $Q=f=g=0$ has
codimension $3$ in $\mathbb{P}^{n-1}$.  Then
$\mu_Q(\mathcal{R}_{f,g})$exists, and is equal to 
$\prod_p \mu_{Q,p}(\mathcal{R}_{f,g})$, where
$$
\mu_{Q,p}(\mathcal{R}_{f,g})=
\hspace{-0.2cm}
\lim_{k\to \infty} 
\frac{\#\{\x\in (\ZZ/p^k\ZZ)^n: Q(\x)\equiv 0\mod{p^k}, ~p\nmid
  \gcd(f(\x),g(\x))\}} 
{
\#\{\x\in (\ZZ/p^k\ZZ)^n: Q(\x)\equiv 0\mod{p^k}\}}.
$$
\end{cor}
Despite having Theorem \ref{MT++} at our disposal, we prove the
corollary only for the case of rank 5 or more, although it seems
likely that it might be extended to cover the quadratic forms in Theorem
\ref{MT++}.

\medskip

A  closely related consequence of the geometric sieve 
concerns  
``arithmetic purity'' for 
projective quadrics. 
The 
implicit function theorem implies  that weak approximation over  $\QQ$
 is birationally invariant among smooth varieties. 
Let $V$ be a variety defined over $\QQ$
such that  $V(\QQ)\neq \emptyset$.
Strong approximation off $\infty$ is said to hold for $V$ if the diagonal image of the set $V(\QQ)$ of rational points is dense in the the space of finite adeles
$V(\A_\QQ^f)$, equipped with the adelic topology. 
Wittenberg 
\cite[Question~2.11]{W}
has asked whether the property of strong approximation off $\infty$
is invariant among smooth varieties up to a closed subvariety
of codimension at least 2.  We say 
$V$ satisfies ``arithmetic purity'' if strong approximation off $\infty$ holds for $V$ and also for the open subset  $V\setminus Z$, for 
any codimension $2$ subvariety $Z\subset V$.
This property  has been observed to hold for $V=\AA^m$ or $V=\PP^m$, for example, 
by Cao and Xu \cite[Prop.~3.6]{CX}.

Smooth projective quadrics with a rational point are well-known to satisfy strong approximation.
The following result establishes the arithmetic purity property for this class of varieties.

\begin{cor}\label{c:SA}
Let  $m\geq 4$ and let $X\subset \PP^m$ be a smooth  quadric hypersurface defined over 
$\QQ$ such that $X(\QQ)\neq \emptyset$.
For
any codimension two subvariety $Z\subset X$ the variety $X\setminus Z$
satisfies strong  approximation off $\infty$. 
\end{cor}

The proof of this result is given in Section \ref{s:SA}. In fact 
Corollary \ref{c:SA} follows rather  easily by adapting 
the proof of Lemma 1.8 in 
  work of Harpaz and Wittenberg \cite{HW}.
     (To be precise, one replaces 
$\AA^n$ by the  quadric $X$ and one replaces the line $L$  passing through $Q$ and $Q'$ by a conic which arises from intersecting $X$ with a plane passing through  $Q$ and $Q'$.) 
We have chosen to include 
Corollary \ref{c:SA} 
 in order to illustrate the scope of the geometric sieve.

 \medskip

Our final application concerns local solubility for families of varieties. 
Recall that a scheme over a perfect
field is said to be split if it contains a geometrically integral open subscheme.
Suppose one has a  family $Y\to X$ of varieties over $\QQ$. A conjecture of Loughran 
 \cite[Conj.~1.7]{Dan} states that under suitable hypotheses, when ordered by height, a positive proportion of the fibres 
have adelic points  if and only if the morphism is split in codimension 1. This is established when $X=\PP^m$ in \cite[Thm.~1.3]{wa}.
The following result confirms the conjecture when $X$ is a   quadric hypersurface of large enough rank.

\begin{cor}\label{c:codim}
Let $X\subset \PP^m$ be a hypersurface defined over $\QQ$ by an indefinite quadratic form of rank at least $5$. 
Let $\pi:Y \to X$ be a dominant quasi-projective $\QQ$-morphism, 
with geometrically
	integral  generic fibre.
 Assume that:
\begin{enumerate}
\item  the fibre of $\pi$ over each  
	codimension-$1$ point of $X$  is split; 
\item $V(\mathbf{A}_\QQ)\neq \emptyset$.
\end{enumerate}
Then the limit
$$
\sigma(\pi)=\lim_{B\to \infty} \frac{
\#\left\{x\in X(\QQ): H(x)\leq B, ~\pi^{-1}(x)(\mathbf{A}_\QQ)\neq \emptyset \right\}
}{\#\left\{x\in X(\QQ): H(x)\leq B \right\}
}
$$
exists, and it is equal to a positive product of local densities. 
\end{cor}

This will be established in Section \ref{s:Ekedahl}, where an explicit  value for  $\sigma(\pi)$ is also recorded.

\subsection*{Acknowledgements}
The authors were inspired to work on this problem following 
discussions at the AIM workshop ``Rational and integral points on
higher-dimensional varieties'' in May, 2014.
They would particularly like to thank David Harari and Olivier 
Wittenberg for their patient explanations of the issues involved 
with the geometric sieve for quadrics.  The authors are also grateful to Julian Lyczak and Olivier Wittenberg for further useful comments. 
During the preparation of this article the first-named author was
supported by EPSRC grant EP/P$026710/1$ 
and FWF grant P 32428-N35.

\section{The geometric sieve for affine space}

We shall reduce the proof of 
Theorem \ref{MT*} to an application of the usual geometric sieve for
affine space. 
However, it will be important  to have a version of
 \cite[Thm.~3.3]{sieve} in  which
the dependence on the coefficients of all the polynomials  is made
explicit and, furthermore,  the variables are allowed to run over a
lopsided box.

Given $B_1,\dots, B_n\geq 1$, it will be convenient to set 
$$
V=\prod_{1\leq i\leq n} B_i
$$
and 
$$
B_{\min}=\min(B_1,\dots,B_n).
$$
We shall adhere to this notation throughout this section, the main
result of which is the  following. 

\begin{lemma}\label{t:ekedahl}
Let $B_1,\dots, B_n,H,M\geq 2$ and let
$f_1,\dots,f_r\in\ZZ[X_1,\dots,X_n]$ be polynomials with no common
factor in the ring $\ZZ[X_1,\dots,X_n]$, and having degrees at most $d$
and heights at most $H$.  Then 
 \begin{align*}
 \#\{\x\in \ZZ^n: |x_i|\leq B_i &\text{ for $i\leq n$}, ~\exists~p>M,\,
 p\mid f_j(\x) \text{ for $j\leq r$} \}  \\
 &\quad \ll 
 \frac{V\log (VH)}{M\log M} + \frac{V \log (VH)}{B_{\min}}, 
 \end{align*}
 where the implied constant is only  allowed to depend  on  $d$ and
 $n$ (and is independent of $r$).
\end{lemma}

Here the height $H(f)$  
of a polynomial $f$ is defined as the maximum of
the moduli of its coefficients.

One recovers a version of 
 \cite[Thm.~3.3]{sieve} by taking $B_1=\dots=B_n$ and by absorbing $H$
 into the implied constant. 
 The proof is a  minor modification of the proof of 
 \cite[Thm.~3.3]{sieve}, but we shall give full details for the sake
 of completeness.   

We begin the proof with an easy lemma.

\begin{lemma}\label{A}
Let  $f\in\ZZ[X_1,\dots,X_n]$ be a non-zero polynomial of degree $d$,
and let $B\geq 1$.  Then
\[\#\left\{\x\in \ZZ^n\cap[-B,B]^n: f(\x)=0
\right\}\leq nd(2B+1)^{n-1} .\]

Moreover, if $p$ is a prime which does not divide $f$ identically, then 
\[\#\left\{\x\in \ZZ^n\cap [-B,B]^n: p\mid f(\x)
\right\}\leq nd(2B/p+1)(2B+1)^{n-1}\]
and
\[\#\left\{\x\in \ZZ^n\cap (0,p]^n: p\mid f(\x)
\right\}\leq ndp^{n-1}.\]
\end{lemma}

\begin{proof}
 The first assertion may be proved by induction on $n$, there being at
most $d$ zeros when $n=1$. For general $n$ suppose that $x_i$ is a
variable that genuinely occurs in $f(\x)$.  
With no loss of generality we may suppose that $i=n$ and that  
$x_n^e$ occurs as
$x_n^ef_0(x_1,\dots,x_{n-1})$ for some exponent $e\le d$, with
$f_0$ not vanishing identically.  By our induction assumption there
are at most 
\[(n-1)d(2B+1)^{n-2}\]
vectors $(x_1,\dots,x_{n-1})\in \ZZ^{n-1}\cap[-B,B]^{n-1}$
which are zeros of $f_0$.
For each of these, there are at most $2B+1$ choices for  $x_n$.
Next, there are at most $(2B+1)^{n-1}$
 choices of $(x_1,\dots,x_{n-1})$ which are not zeros of $f_0$, 
 and for each of these there are at most $d$ possible
values for $x_n$.  The total number of solutions is thus at most
\[ (n-1)d(2B+1)^{n-1}+d(2B+1)^{n-1}= nd(2B+1)^{n-1}.\]
This completes the induction step.

For the second assertion we argue similarly, supposing that $x_n^e$
occurs in $f$ as $x_n^ef_0(x_1,\dots,x_{n-1})$ with $f_0$ not 
identically divisible by $p$. The argument then proceeds as before,
except that now a non-trivial polynomial congruence in one variable
$x$, of degree at most $d$, has at most $d(2B/p+1)$ solutions modulo $p$ in the
interval $[-B,B]$. The final claim is proved similarly, a one-variable
congruence having at most $d$ solutions.
\end{proof}

We now start the proof of Lemma \ref{t:ekedahl}. When $r=1$ the
coprimality condition means that $f_1$ must be constant, equal to
$\pm 1$. In this case there can never be a prime $p>M$ dividing
$f_1$.  We may therefore assume from now on that $r$ is at least
2, and our first move is to show that it suffices to take $r=2$.
Let us temporarily write $\cN(f_1,\dots,f_r)$ for the counting
function in Lemma \ref{t:ekedahl}.  If $f_1$ factors into irreducibles
as $g_1\dots g_k$ over $\ZZ[X_1,\dots,X_n]$ one sees that $k\le d$ and
\[\cN(f_1,\dots,f_r)\le\sum_{j=1}^k\cN(g_j,f_2,\dots,f_r).\]
Each polynomial $g_j$ will have degree at most $d$.  Moreover, for
any polynomials $u,v \in \RR[X_1,\dots,X_n]$ 
with degree at most
$d$ one has 
\[H(u)H(v)\ll_{n,d}H(uv),\]
by Prasolov  \cite[Section 4.2.4]{prasolov}, for example. It follows that
$H(g_j)\ll_{n,d}H$, and one then sees that it will suffice to prove
the lemma in the case in which $f_1$ is irreducible. With this
latter assumption the coprimality condition shows that not all of
$f_2,\dots,f_r$ can be divisible by $f_1$.  We suppose without loss
of generality that $f_1\nmid f_2$, and note that
\[\cN(f_1,\dots,f_r)\le\cN(f_1,f_2),\]
with $f_1$ and $f_2$ coprime. Thus it suffices to prove the lemma in
the case $r=2$, as claimed.

We proceed to make a further simplification, reducing to the case in
which $B_1=\dots=B_n$.  To achieve this, set $k=[B_{\min}]$. Then
if $|x_i|\le B_i$ we may write $x_i=y_i+kh_i$ with $0\le y_i<k$ and
$|h_i|\le 1+B_i/k\ll B_i/B_{\min}$. We set
$f_i(\b{Y};\b{h})=f_i(\b{Y}+k\b{h})$, and observe that these will
be coprime as polynomials in
$\b{Y}$, for any fixed $\b{h}$.  Moreover they will
have height at most $O_{d,n}(HV^d)$.  Thus if we have proved Lemma
\ref{t:ekedahl} in the case $B_1=\dots=B_n(=k)$, we may deduce that
the number of acceptable vectors $\y$ corresponding to a given choice
of $\b{h}$ will be
\[\ll_{n,d}\frac{k^n\log(VH)}{M\log M}+\frac{k^n\log(VH)}{k}.\]
Since there are $O_{n,d}(VB_{\min}^{-n})$ choices for $\b{h}$ we then
recover the required bound for general lopsided values of the $B_i$.

For the remainder of the proof we may now assume that
$B_1=\dots=B_n=B$, say, so that we need to prove that the number of
suitable $\x$ is
\beql{n2}
\ll_{n,d}\frac{B^n\log (BH)}{M\log M} + B^{n-1}\log (BH).
\eeq

We have one further manoeuvre to perform
before reaching the crux of the proof, and that is to show that we may
assume that if $f_1f_2$ has total degree $e(\le 2d)$ then $f_1f_2$ contains a
non-zero term in $X_1^e$.  (Hence both $f_1$ and $f_2$ will
contain monomials in $X_1$ of the maximum possible degrees.)
To show this, let $F(\b{X})$ be the homogeneous
part of $f_1(\b{X})f_2(\b{X})$ of degree $e$.
According to Lemma \ref{A}, the form $F$ has at most $ne(2K+1)^{n-1}$ zeros with
$|\x|\le K$. Taking $K=ne$ we deduce that there is a non-zero integer vector
$\b{a}$ with $F(\b{a})\not=0$, having size $|\b{a}|\le ne$. Without
loss of generality we will suppose that $a_1\not=0$. We now define
variables $Y_i$ by setting $Y_1=X_1$, and $Y_i=a_1X_i-a_iX_1$ for
$2\le i\le n$. We then have $a_1X_1=a_1Y_1$, and $a_1X_i=a_iY_1+Y_i$
for $2\le i\le n$. Then $a_1^df_j(\b{X})$ may be written as $g_j(\b{Y})$ 
say, for $j=1,2$, with $H(g_j)\ll_{d,n} H$. Moreover the coefficient
of $Y_1^e$ in $g_1g_2$ will be $a_1^{2d-e}F(\b{a})\not=0$. We also see
that $\b{y}$
is an integer vector whenever $\x$ is, and that $|\y|\ll_{n,d}B$
whenever $|\x|\le B$.  The linear transform connecting $\b{X}$ and $\b{Y}$
has determinant $a_1^{n-1}$, so that any constant factors of $g_1(\b{Y})$ or
$g_2(\b{Y})$ must have prime factors dividing $a_1$.  These may
safely be removed, since Lemma \ref{t:ekedahl} is trivial when
$M\ll_{n,d}1$. We then see that it suffices to prove the lemma for the
polynomials $g_1$ and $g_2$.
\medskip

We now proceed with the proof, under the assumption that
\[B_1=\dots=B_n=B,\]
and that $f_1f_2$ has a non-zero term,
$cX_1^e$ say, where $e$ is the total degree of $f_1f_2$.
We begin by considering the case in which
there is a prime $p>M$ dividing both $f_1(\x)$ and $f_2(\x)$ and for which
$p\mid c$. Since $c\ll_{n,d}H^2$, the number of such primes is
$O_{n,d}(\log H/\log M)$. 
It is not possible for
both $f_1(\b{X})$ and $f_2(\b{X})$ to vanish 
modulo $p$, since we have assumed that $f_1$ and $f_2$
have no constant factor. Assume without loss of generality that 
$f_1(\b{X})$ does not vanish modulo $p$. 
We may therefore apply Lemma \ref{A}, which shows that 
the number of possible
$\x$ for which $p\mid f_1(\x)$ will be
\[\ll_{n,d}(B/p+1)B^{n-1}\ll B^nM^{-1}+B^{n-1}.\]
This is satisfactory for (\ref{n2}), since the number of available
  primes is 
  \[\ll_{n,d}\frac{\log H}{\log M}.\]

We next consider primes which do not divide $c$. Let
$R(X_2,\dots,X_n)$ be the resultant $\mathrm{Res}_{X_1}(f_1,f_2)$ of $f_1$
and $f_2$ with
respect to $X_1$. Since $f_1$ and $f_2$ are coprime over
$\ZZ[X_1,\dots,X_n]$ this resultant cannot vanish identically.
If $f_1$ and $f_2$ have degrees $d_1$ and $d_2$ with
respect to $X_1$ this resultant is 
given by the determinant of a $(d_1+d_2)\times (d_1+d_2)$ matrix, 
whose entries are polynomials
in $X_2,\dots,X_n$, of height $O_{n,d}(H)$ and degree at most
$d$. Thus $R$ has degree at most $2d^2$ and height
$H(R)\ll_{n,d}H^{2d}$.  Moreover, for any choice
$(x_2,\dots,x_n)\in\ZZ^{n-1}$, the 1-variable
polynomials $f_1(X_1,x_2,\dots,x_n)$ and $f_2(X_1,x_2,\dots,x_n)$
have a common factor modulo $p$ if and only if 
$p\mid R(x_2,\dots,x_n)$. Note that for us to draw this conclusion we
need to observe that the 1-variable
polynomials $f_1(X_1,x_2,\dots,x_n)$ and $f_2(X_1,x_2,\dots,x_n)$ still
have degrees $d_1$ and $d_2$ when considered modulo $p$, because
$p\nmid c$. There are now two alternative situations to
consider. Firstly, it could happen that $R(x_2,\dots,x_n)=0$.
According to Lemma \ref{A} there are at most $O_{n,d}(B^{n-2})$
possible solutions $(x_2,\dots,x_n)\in\ZZ^{n-1}$ in the cube
$[-B,B]^{n-1}$. For each of these there are at most $2B+1$
possibilities for $x_1$, making $O_{n,d}(B^{n-1})$ in total.  This
is acceptable for (\ref{n2}).  In the alternative case we have
$R(x_2,\dots,x_n)\not=0$. If $p$ divides both
$f_1(x_1,x_2,\dots,x_n)$ and $f_2(x_1,x_2,\dots,x_n)$ then the
1-variable polynomials $f_1(X_1,x_2,\dots,x_n)$ and
$f_2(X_1,x_2,\dots,x_n)$ have a common root modulo $p$, namely
$x_1$. We must therefore have $p\mid R(x_2,\dots,x_n)$.  Since $R$
has degree at most $2d^2$ and height $O_{n,d}(H^{2d})$, 
we have
\[0<|R(x_2,\dots,x_n)|\ll_{n,d}B^{2d^2}H^{2d}.\]
It follows that the number of primes $p>M$ which can divide
$R(x_2,\dots,x_n)$ is $O_{n,d}((\log BH)/(\log M))$. Given
$x_2,\dots,x_n$, and given a prime $p\mid R(x_2,\dots,x_n)$, there
are at most $2B/p+1\le 2B/M+1$ integers $x_1\in[-B,B]$ for which
$p$ divides $f_1(x_1,x_2,\dots,x_n)$, by Lemma \ref{A}. Here we note that
the 1-variable polynomial $f_1(X_1,x_2,\dots,x_n)$ does not
vanish modulo $p$, since $p\nmid c$.  We now deduce that there are
\[\ll_{n,d}B^{n-1}\frac{\log(BH)}{\log M}\left(\frac{B}{M}+1\right)\]
vectors $(x_1,\dots,x_n)\in\ZZ^n\cap[-B,B]^n$
for which $R(x_2,\dots,x_n)\not=0$ and such that
$f_1(x_1,x_2,\dots,x_n)$ and $f_2(x_1,x_2,\dots,x_n)$ have a common
factor $p>M$ which does not divide $c$. This bound is again
acceptable, thereby completing our treatment of (\ref{n2}).

\section{The geometric sieve for quadrics:  preliminaries}\label{s:quadric}

We will deduce Theorem \ref{MT+} from a result in which the quadric
takes a specific shape.

\begin{theorem}\label{MT}
Let 
$Q_0(X_2,\dots,X_n)$ be a quadratic form defined over $\ZZ$ and let
$F_1(X_0,\dots, X_n),\dots,F_r(X_0,\dots, X_n)$
be forms defined over $\ZZ$. Write
\[Q(X_0,\dots,X_n)=X_0X_1-Q_0(X_2,\dots,X_n).\]
Assume that the rank of $Q_0$ is at least
$3$ and that the variety $Z\subset\mathbb{P}^{n}$ given by
\[Z:\, Q(X_0,\dots,X_n)=F_1(X_0,\dots, X_n)=
\dots=F_r(X_0,\dots, X_n)=0\]
has codimension at least 3 in $\mathbb{P}^n$.
For $B,M\ge 1$ let  $N(B,M)$  be the number of vectors
$\x\in\ZZ^{n+1}$ such that 
\begin{equation}\label{eq:Q}
Q(x_0,\dots,x_n)=0,
\end{equation}
with $|\x|\leq B$, and for which 
$F_1(x_0,\dots,x_n),\dots,F_r(x_0,\dots,x_n)$ have a
common prime divisor $p> M$. Then
\[N(B,M)\ll_{\ve, Q,F_1,\dots,F_r} \frac{B^{n-1+\ve}}{M\log M}+B^{n-1-1/n+\ve},\]
for any fixed $\ep>0$.
\end{theorem}

Let us show how this result implies Theorem \ref{MT+}.
We first note that if the quadric hypersurface has no
non-singular rational points (i.e.\  if $Q$ is not indefinite)
the rational points will be restricted to a linear space of dimension
$n-\rank(Q)$. In this case there will only be $O_n(B^{n-4})$
rational points of height $B$ or less. This is more than sufficient,
and so we may assume that there is at least one
smooth rational point. In this case there is a linear transformation
$\mathbf{T}_1\in\mathrm{SL}_{n+1}(\QQ)$ such that
\[Q(\mathbf{T}_1^{-1}\b{X})=X_0X_1-Q_1(X_2,\dots,X_n), \]
where $Q_1$ in a quadratic form with rational coefficients. Rescaling
the variables $X_2,\dots,X_n$ we obtain $\mathbf{T}_2\in\mathrm{GL}_{n+1}(\QQ)$
such that $Q(\mathbf{T}_2^{-1}\b{X})=Q^*(\b{X})$, where
\[Q^*(\b{X})=X_0X_1-Q_0(X_2,\dots,X_n), \]
with $Q_0\in\ZZ[X_2,\dots,X_n]$. We then have
$Q^*(\mathbf{T}_2\b{X})=Q(\b{X})$. We now choose $N$ so that $N\mathbf{T}_2=\mathbf{T}$ has
integer entries, with the result that $\mathbf{T}\b{x}$ is an integer zero of $Q^*$
whenever $\b{x}$ is an integer zero of $Q$. We can choose $\mathbf{T}$ to
depend only on $Q$, so that $|\mathbf{T}\b{x}|\ll_Q |\x|$. Finally, if the forms
$F_i$ have degrees at most $d$, and we set
$G_i(\b{X})=\det(\mathbf{T})^dF_i(\mathbf{T}^{-1}\b{X})$, then the forms $G_i$ will have
integer coefficients, and any common prime divisor of
$F_1(\x),\dots,F_r(\x)$ will also divide
$G_1(\mathbf{T}\x),\dots,G_r(\mathbf{T}\x)$. Since the variety $Q^*=G_1=\dots=G_r=0$ is
produced from $Q=F_1=\dots=F_r=0$ by a non-singular linear
transformation, it also has codimension at least 3 in $\mathbb{P}^n$.
We therefore see that Theorem \ref{MT} applies to $Q^*$ and
$G_1,\dots,G_r$, and yields exactly the bound required for Theorem \ref{MT+}.
\medskip

We now begin our treatment of Theorem \ref{MT}. For the proof we shall
allow all of our implied constants to depend on
the polynomials $Q,F_1,\dots,F_r$, as well as on the small parameter
$\ve>0$. We begin by disposing of points on the quadric (\ref{eq:Q})
for which there is a prime
$p>M$ dividing $x_0$ and $x_1$ as well as 
$F_1(x_0,\dots,x_n),\dots,F_r(x_0,\dots,x_n)$.
In this case $p^2$
divides $Q_0(x_2,\dots,x_n)$, so that $Q_0(x_2,\dots,x_n)=p^2k$ for some
integer $k\ll B^2p^{-2}$.  The equation $Q_0(x_2,\dots,x_n)=h$ has
$O(B^{n-3+\ep})$ integer solutions in $[-B,B]^{n-1}$, uniformly in $h$.
(This would be false for $h=0$ if $Q_0$ had rank at most 2 and
factored over $\mathbb{Q}$.)  Moreover the equation $x_0x_1=h$ has
$O(B^{\ve})$ solutions when 
$h\not=0$. The case $k\not=0$ therefore produces a contribution
\[\ll\sum_{p>M}B^2p^{-2}.B^{n-3+\ep}.B^{\ep}\ll \frac{B^{n-1+2\ep}}{M\log M}.\]
On the other hand, the equation $x_0x_1=0$ has $O(B)$ solutions of the
correct size, so that the case $k=0$ contributes $O(B^{n-3+\ep}.B)$
solutions. Hence, on re-defining $\ep$ we see that the number of
points under consideration is
\[\ll \frac{B^{n-1+\ep}}{M\log M}+B^{n-2+\ep}.\]
This is satisfactory for the theorem.

We may now assume that the common prime factor of
\[F_1(x_0,\dots,x_n),\dots,F_r(x_0,\dots,x_n)\]
does not divide both
$x_0$ and $x_1$, and we proceed to estimate $N_i(B,M)$,
defined for $i=0,1$ to be 
the number of vectors $\x\in\ZZ^{n+1}$ on the quadric (\ref{eq:Q})
such that $|\x|\leq B$, and for which  
$F_1(x_0,\dots,x_n),\dots,F_r(x_0,\dots,x_n)$ have a
common prime divisor $p> M$ which does not divide $x_i$. Clearly it
will now suffice to estimate both $N_0(B,M)$ and $N_1(B,M)$. By
symmetry, it will be enough to consider $N_1(B,M)$.

We may add suitable multiples of 
$Q$ to any of the forms $F_i$, so as to
suppose that $F_i$ has no monomials divisible by $X_0X_1$.
This will not affect the hypotheses of Theorem \ref{MT}.
If all the $F_i$ have degrees at most $D$ we may then write
\[F_i(X_0,\dots,X_n)=G_i(X_1,\dots,X_n)+
\sum_{j=1}^DX_0^jH_{i,j}(X_2,\dots,X_n),\]
say. Then if $(x_0,\dots,x_n)$ lies on the quadric (\ref{eq:Q}) we
will have
\[x_1^DF_i(x_0,\dots,x_n)=K_i(x_1,\dots,x_n),\]
with
\begin{align*}
  K_i(X_1,\dots,X_n)=X_1^D&G_i(X_1,\dots,X_n)\\
  &+
  \sum_{j=1}^DQ_0(X_2,\dots,X_n)^jX_1^{D-j}H_{i,j}(X_2,\dots,X_n).
  \end{align*}
Thus if $p\mid F_i$ for all $i$, then $p\mid K_i$ for all $i$.

We now claim that the forms $K_i$ can have no common factor of
positive degree over $\overline{\QQ}[X_1,\dots,X_n]$, except possibly a power of $X_1$. Suppose for a
contradiction that $R(X_1,\dots,X_n)$ is an irreducible form,
different from $X_1$, which divides all the forms $K_i$, so that
$K_i=RS_i$, say.  It is clear from our construction that we may write
\[K_i(X_1,\dots,X_n)=X_1^DF_i(X_0,\dots,X_n)
+Q(X_0,\dots,X_n)T_i(X_0,\dots,X_n)\]
for suitable forms $T_i$, so that
\[RS_i=X_1^DF_i+QT_i.\]
We then see that any point on $Q=R=0$ lies either on $Q=X_1=0$ or on
$Q=F_1=\dots=F_r=0$.  However every irreducible component of the
intersection $Q=R=0$ has codimension at most 2 in $\mathbb{P}^n$,
while the variety $Q=F_1=\dots=F_r=0$ was assumed to have codimension
at least 3. It follows that the intersection $Q=R=0$ must be contained
in the hyperplane $X_1=0$. This however is impossible. Indeed, since $X_1$
does not divide $R$ there are
points on $R=0$ for which $x_1\not=0$, and since $R$ does not involve
$X_0$ we can choose $x_0$ so that $Q=0$ as well. This gives a point of
$Q=R=0$ not lying on the hyperplane $X_1=0$.  This contradiction
proves our claim.

\section{The geometric sieve for quadrics: lattices}
\label{s:quadricL}
We now wish to count points on $Q=0$, such that the 
forms $K_i$ have a common factor $p>M$ that does not divide $x_1$. We
have arranged that the $K_i$ do not involve $X_0$, and that they have
no common factor of positive degree except possibly for powers of
$X_1$. We may remove any such factors, since they will not affect the
divisibility by $p$. Indeed we may remove any constant factors, since
Theorem \ref{MT} is trivial when $M\ll_{\ve, Q,F_1,\dots,F_r}1$,
because the quadric (\ref{eq:Q}) has $O(B^{n-1+\ep})$ points.

Our plan is to apply the geometric sieve for $\mathbb{A}^n$ to the
$K_i$, but we need to account for the condition that
$Q_0(x_2,\dots,x_n)=x_0x_1$. We may eliminate any mention of the
variable $x_0$ by weakening this last condition to say instead that
$x_1\mid Q_0(x_2,\dots,x_n)$.  In effect we then need a geometric
sieve for $\mathbb{A}^n$, with a divisibility side condition. We
tackle this problem by fixing $x_1$, and working with
$(x_2,\dots,x_n)\in\mathbb{A}^{n-1}$, subject to a divisibility
condition for a modulus $x_1$, which is now fixed. The key idea is
then to interpret this divisibility condition in terms of lattices.

It will be notationally convenient to work with a general quadratic
 form $R(X_1,\dots,X_m)$ of rank at least 3, in place of 
$Q_0(X_2,\dots,X_n)$.  We shall say that a prime is ``$R$-good'' if
it is odd and the 
reduction of $R$ modulo $p$  has the same rank as $R$ itself. Let $q$ be a
product of distinct $R$-good primes. We seek to cover all integer
vector solutions of the congruence $R(x_1,\dots,x_m)\equiv 0\mod{q}$ by
lattices of the shape
\beql{sLd}
\sL(\y):=\{\x\in \ZZ^m:\,\exists \rho\in \ZZ,\,\x\equiv \rho \y\mod{q}\},
\eeq
for suitable $\y\in \ZZ^m$ with $\gcd(\y,q)=1$. We note  that 
 $\sL(\y)$  has rank $m$ and  determinant $q^{m-1}$.
We begin by asking how many such lattices will be required.  

\begin{lemma}\label{LL}
  Suppose that $R\in\Z[X_1,\dots,X_m]$ is a quadratic form of
  rank at least $3$, and let $q\in\N$ be a product of distinct $R$-good
  primes. Then
\[\left\{\x\in \ZZ^m: R(\x)\equiv 0\mod{q}\right\}\subseteq 
\bigcup_{\y\in Y(q)}\sL(\y),\]
where $\sL(\y)$ is given by \eqref{sLd} and
$\card Y(q)\le (3m)^{\omega(q)}q^{m-2}$.  Moreover, 
 each $\y\in Y(q)$ is an integer vector satisfying
$Q(\y)\equiv 0\mod{q}$ and $\gcd(\y,q)=1$. 

 Finally, for any $L>0$, the number of these lattices for which the
 largest successive minimum is greater than $L$, is
\[\ll_m (3m)^{\omega(q)}q^{2m-3}L^{-m}.\]
\end{lemma}

Note that our successive minima are taken with respect to the Euclidean
norm $||\cdot||_2$.

\begin{proof}
For the first part it is enough to consider the individual prime
factors of $q$, and to 
combine the corresponding lattices using the Chinese Remainder Theorem.
Assume that  $q=p$ is an $R$-good prime.  According to the final part
of Lemma \ref{A} the congruence $R(\x)\equiv 0\mod{p}$ has at most
$2mp^{m-1}$ solutions. (This is a very poor bound, but sufficient for
our purposes.)
The solutions $\x\not\equiv\mathbf{0}\mod{p}$ will then be covered by at most
\[2mp^{m-1}/(p-1)\le 3mp^{m-2}\]
lattices $\sL(\y)$ with $Q(\y)\equiv 0\mod{p}$ and
$p\nmid\y$. Since $m\ge 3$ there is at least one such $\y$, and the
corresponding lattice will cover the solution $\mathbf{0}$.
It then follows that for general $q$ we can cover all solutions using
at most $(3m)^{\omega(q)}q^{m-2}$
lattices $\sL(\y)$ with $R(\y)\equiv 0\mod{q}$ and $\gcd(\y,q)=1$. 

Associated to  any rank $m$ lattice $\sL\subset\R^m$ is the 
dual  lattice 
$$
\sL^*=\{\mathbf{t}\in\R^m: \mathbf{t}.\x\in\Z, \forall \x\in\sL\}.
$$
If  the successive minima of 
$\sL$ are $\lambda_1\leq \dots \leq \lambda_m$, and 
  the successive minima of the dual lattice
$\sL^*$ are $\lambda_1^*\leq \dots \leq \lambda_m^*$, then it follows
from Theorem VI on page 219 of Cassels \cite{cassels} that 
$$
1\leq \lambda_i \lambda_{m+1-i}^* \leq m! ,
$$
for $1\leq i\leq m$. We shall apply this with $\sL=\sL(\y)$. 
Assume that  $\lambda_m>L$.
Then it follows that $\lambda_1^* \leq m! L^{-1}$.
Since $q\Z^m\subseteq\sL(\y)\subseteq\Z^m$, it follows that 
$\Z^m\subseteq\sL(\y)^*\subseteq q^{-1}\Z^m$. Each element of
$\sL(\y)$ has the shape 
 $\rho\y+q\b{k}$ for some $\b{k}\in\Z^m$, so that
$q^{-1}\b{s}$ belongs to $\sL(\y)^*$ if and only if $\b{s}$ is an integer
vector for which $q^{-1}\rho \b{s}.\y\in\Z$ for every $\rho\in \ZZ$. 
But this is  equivalent to 
 $\b{s}$ being an integer
vector for which $\b{s}.\y\equiv 0 \mod{q}$. 
Thus $q\lambda_1^*$ will be the
length of the shortest non-zero integer vector for which
$\b{s}.\y\equiv 0\mod{q}$. It follows that if $\sL(\y)$ has
$\lambda_m> L$ then $\b{s}.\y\equiv 0\mod{q}$ for some non-zero
integer vector $\b{s}$ with $|\b{s}|\leq m! q/L$.

We now bound the number of lattices with $\lambda_m>L$.
Here we should recall that the total number of
lattices $\sL(\y)$ under consideration is at most $(3m)^{\omega(q)}q^{m-2}$.
For each choice of $\b{s}$ we count values of $\y$ modulo
$q$ for which both $R(\y)\equiv 0\mod{q}$ and $\b{s}.\y\equiv 0\mod{q}$.
This can be done by applying the Chinese Remainder
Theorem to the case in which $q=p$ is a prime. The vector $\b{s}$ need
not be primitive, and if $p\mid\b{s}$ there will be at most
$2mp^{m-1}$ values of $\y$, as above. On the other hand, when
$p\nmid\b{s}$ the conditions produce a non-trivial hyperplane slice of
the quadric $R=0$ over $\mathbb{F}_p$. Since the prime $p$ is $R$-good
the form $R$ has rank at least 3 over $\mathbb{F}_p$. It follows that
the hyperplane cannot contain the quadric, whence Lemma \ref{A} shows that 
there are at most $2(m-1)p^{m-2}$ solutions $\y$, corresponding to at
most
\[\frac{2(m-1)p^{m-2}}{p-1}\le 3mp^{m-3}\]
points in $\mathbb{P}^{m-1}(\mathbb{F}_p)$.  It then
follows from the Chinese Remainder Theorem that there are at most
$(3m)^{\omega(q)}q^{m-3}\gcd(q,\b{s})$
distinct lattices corresponding to $\b{s}$.
We may now sum over non-zero integer
vectors $\b{s}$ with $|\b{s}|\leq m!q/L$. When $\gcd(q,\b{s})=d$, say,
there are no such $\b{s}$ unless $d\leq m!q/L$, in which case there
will be at most $\ll_m q^mL^{-m}d^{-m}$ possible vectors $\b{s}$.
This gives a total contribution
\[\ll_m (3m)^{\omega(q)}q^{m-3}d\cdot q^mL^{-m}d^{-m},\]
for each divisor $d$ of $q$. Since $m\ge 3$ we may then sum over
$d\mid m$ to produce the bound stated in the lemma.
\end{proof}

We are now ready to put our plan into action. Recall that we are
counting points $(x_1,\dots,x_n)\in\ZZ^n$ of size at most $B$, such
that $x_1$ is non-zero and is a divisor of $Q_0(x_2,\dots,x_n)$, and for which
$K_1(x_1,\dots,x_n),\dots,K_r(x_1,\dots,x_n)$ have a common prime
factor $p>M$ which does not divide $x_1$.

We take $q=q(x_1)$ to be the product of
all $Q_0$-good primes dividing $x_1$, and we weaken the condition
$x_1\mid Q_0(x_2,\dots,x_n)$, requiring instead only that
$q\mid Q_0(x_2,\dots,x_n)$.  We apply Lemma \ref{LL} to the form
$R=Q_0$, in $m=n-1$ variables.  The corresponding lattices $\sL(\y)$
are therefore contained in $\ZZ^{n-1}$. The lemma then shows that 
\[N_1(B,M)\leq
\sum_{q\le B}\sum_{\substack{a\neq 0\\ q(a)=q}}\sum_{\y\in Y(q)}N(B,M,q,\y,a),\]
where $N(B,M,q,\y,a)$ is the number of 
$\x=(x_2,\dots,x_n)\in\sL(\y)$ in the box $|\x|\le B$ for which
the polynomials $K_i(a,x_2,\dots,x_n)$ all have a common prime divisor
$p> M$. Notice that we have written $a$ in place of $x_1$ to emphasize
the different role it plays in our argument.

We proceed to estimate how many values of $a$ can correspond
to a given $q$.  Let $\Delta$ be the product of the (finitely many)
primes which are not $Q$-good. Then $q$ will divide $a$ and every
prime factor of $a/q=t$ will divide $\Delta q$.  Since we will have
$|t|\le B$ we find using Rankin's trick that the
number of available $t$ is at most
\begin{align*}
  2\sum_{\substack{1\le t\le B\\ t\mid (\Delta q)^{\infty}}}1&\le 
  2\sum_{\substack{t\mid (\Delta q)^{\infty}}}
  \frac{B^{\ep}}{t^{\ep}}\\
    &= 2B^{\ep}\prod_{p\mid\Delta q}\frac{1}{1-p^{-\ep}}\\
    &\ll B^{\ep}\tau(\Delta q)\\
    &\ll B^{2\ep},
\end{align*}
whenever $\ep>0$. Here we have used the fact that $q\le B$ at the very
last step. On   re-defining $\ep$, we therefore see
 that for every $q$ there is a value $a^{(q)}$
 which is divisible by $q$, 
 such that
\beql{N1BM}
N_1(B,M)\ll B^\ve\sum_{q\le B}\sum_{\y\in Y(q)}N(B,M,q,\y,a^{(q)}).
\eeq

Suppose now that we have a lattice $\sL=\sL(\y)$ with $\y\in Y(q)$.
As previously, suppose that $\lambda_1\le\dots\le\lambda_m$ are the 
successive minima of $\sL$, which we recall has determinant $q^{m-1}$.
(Here we continue to use the notation $m=n-1$ for the dimension of $\sL(\y)$.)
It follows from Minkowski's second convex body theorem
\cite[Section~VIII.2]{cassels}
 that
\beql{lp}
q^{m-1}\leq \prod_{i=1}^m \lambda_i\ll_m q^{m-1}. 
\eeq
Moreover, it is clear that $\sL$ has $m$ independent vectors
of length $q$, so that $\lambda_m\le q$.  
According to the corollary to Theorem VII on page 222 of
Cassels \cite{cassels}, the lattice $\sL$
has a basis $\b{e}_1,\dots,\b{e}_m$ with $|\b{e}_j|\ll\lambda_j$ for
all $j$.  We now define
$\mathbf{E}$ to  be the $m\times m$ matrix formed by the column
vectors $\b{e}_1, \dots ,\b{e}_m$. 
Then the maximum modulus of the entries of $\mathbf{E}$ is
\[
||\mathbf{E}||\ll \lambda_m\le q\le B.
\]
Moreover, 
$|\det(\mathbf{E})|=\det(\sL)=q^{m-1}$.
We then see that 
$\mathbf{E}^{-1}$ is the transpose of the matrix formed from column
vectors $\b{e}_1^*,\dots, \b{e}_m^*$, say, where 
\begin{align*}
  |\b{e}_j^*|&\ll
  |\det(\mathbf{E})|^{-1}\prod_{\substack{i=1\\ i\not =j}}^m|\b{e}_i|\\
  &\ll |\det(\mathbf{E})|^{-1}\prod_{\substack{i=1\\ i\not =j}}^m\lambda_i\\
&\ll \lambda_j^{-1},
\end{align*}
by \eqref{lp}. Moreover, as described in \cite[Section I.5]{cassels}, we have 
$$
\e_j^*.\e_i=\begin{cases}
1 &\text{ if $i=j$,}\\
0 &\text{ otherwise.}
\end{cases}
$$
Thus if $\x\in\sL$ is
written as $\x=w_1\b{e}_1+\dots+w_m\b{e}_m$, we will have
$w_j=\b{e}_j^*.\x$, so that 
$w_j\ll|\x|/\lambda_j$ for each index $j$.  

The next stage of the argument is to handle those $\y\in Y(q)$ for
which one has $\lambda_m>L$. Since we automatically have
$\lambda_m\le q\le B$,
it follows from the above that 
the number of $\x=(x_2,\dots,x_n)\in\sL(\y)$ in
the box $|\x|\le B$ will be
\[
\ll \prod_{1\leq j\leq m} \left(B/\lambda_j +1\right)
\ll B^mq^{1-m}=B^{n-1}q^{2-n}. \]
Thus we will trivially have
$N(B,M,q,\y,a^{(q)})\ll B^{n-1}q^{2-n}$. Combining this with the estimate in
Lemma \ref{LL} for the number of lattices with $\lambda_n>L$
we find that the contribution to $N_1(B,M)$ is
\[\ll B^{\ep}\sum_{q\le B}(3n)^{\omega(q)}q^{2n-5}L^{1-n}B^{n-1}q^{2-n}\ll
B^{2n-3+2\ep}L^{1-n}.\]
On re-defining $\ep$, we therefore conclude that 
\begin{equation}\label{need}
  N_1(B,M)\ll B^{2n-3+\ve}L^{1-n}+B^\ve
  \sum_{q\le B}\sum_{\substack{\y\in Y(q)\\ \lambda_m\leq L}}
N(B,M,q,\y,a^{(q)}).
\end{equation}

Suppose now that $\sL=\sL(\y)$ is a lattice  with $\y\in Y(q)$, and for
which $\lambda_m\leq L$. We define polynomials
\[f_i(W_1,\dots,W_m)=K_i\left(a^{(q)},\mathbf{E}\b{W}\right)\;\;\;
(1\le i\le r),\]
where $\mathbf{W}$ is the column vector $(W_1,\dots, W_m)$ and
$\mathbf{E}$ is the matrix defined above, formed from the basis vectors for
$\sL$. We are then left with estimating
the number of integer vectors
$\b{w}\in\ZZ^m$,  with $w_j\ll B/\lambda_j$ for $1\leq j\leq m$,  and 
for which all the $f_i(\b{w})$ have a prime factor $p> M$ in common,
for which $p\nmid a^{(q)}$. We already observed that the forms $K_i$ can be
taken to have no common factor, and we now claim that the polynomials
$f_i$ can have no common factors apart possibly for primes $p$ that
divide $a^{(q)}$. To see this, suppose firstly that $g(W_1,\dots,W_m)$ is a
non-constant common factor of the $f_i$, with $f_i=gh_i$, say.
We then set $W_i=U_iU_0^{-1}$ and multiply through by $U_0^{d_i}$, 
where $d_i$ is the degree of $f_i$.  This will produce relations
\[K_i\left(a^{(q)}U_0,\mathbf{E}\b{U}\right)=G(U_0,\dots,U_m)
H_i(U_0,\dots,U_m)\]
in which $G$ and the $H_i$ are homogeneous, and $G$ is
non-constant. After a non-singular linear change of variables one would then find a
common factor of the forms $K_i(X_1,\dots,X_n)$, at least over
$\QQ[X_1,\dots,X_n]$. This contradiction shows that the $f_i$ cannot
have a non-constant common factor.  Suppose now that there is a prime
common factor $p\nmid a^{(q)}$. It is then clear that $p$ must divide
the forms $K_i\left(W_0,\mathbf{E}\b{W}\right)$. However, since
$\mathbf{E}$ has determinant $q^{m-1}$, with $q\mid a^{(q)}$, it must be
invertible modulo $p$.  It would then follow that $p$ divides each of
the forms $K_i(X_1,\dots,X_n)$, which is impossible.

Since we are concerned with common prime factors $p>M$ which do not
divide $a^{(q)}$ we may remove from the polynomials $f_i$ any constant
factors dividing $a^{(q)}$.  The situation is then exactly right for
an application of Lemma \ref{t:ekedahl}. We note that the polynomials
$f_i$ have height bounded by a power of $B$, so that the lemma yields
\[N(B,M,q,\y,a^{(q)})
\ll\frac{V\log B}{M\log M}+\frac{V\log B}{B_{\mathrm{min}}},\]
where
\[V=\prod_{i=1}^m(B/\lambda_i)\ll B^mq^{1-m}\]
by (\ref{lp}), and $B_{\mathrm{min}}=B/\lambda_m\ge B/L$. Here we have used
the observation that $\lambda_i\le q\le B$ for each index $i$, so
that $B/\lambda_i\gg 1$.  Recalling that $m=n-1$ it follows that
\[N(B,M,q,\y,a^{(q)}) \ll B^{n-2}q^{2-n}\left\{\frac{B}{M\log M}+L\right\}\log B.\]
We proceed to insert this estimate into (\ref{need}), using the bound
for $\card Y(q)$ given by Lemma \ref{LL}.  This produces
\begin{align*}
  N(B,M)&\ll B^{2n-3+\ve}L^{1-n}\\
 & \hspace{1cm}+B^\ve
  \sum_{q\le B}(3n)^{\omega(q)}q^{n-3}\cdot B^{n-2}q^{2-n}\left\{
  \frac{B}{M\log M}
  +L\right\}\log B\\
  &\ll B^{2n-3+\ve}L^{1-n}+B^{n-2+2\ve}\left\{
  \frac{B}{M\log M}
+L\right\}.
\end{align*}
We therefore choose $L=B^{1-1/n}$, and Theorem \ref{MT} follows, on
re-defining $\ep$.

\section{Proof of Theorem \ref{MT++}}

Our argument starts in the same way as for Theorem \ref{MT+} in
Section \ref{s:quadric}. As before we may assume that
$Q(\b{X})=X_0X_1-Q_0(X_2,\dots,X_n)$ with
\[Q_0(X_2,\dots,X_n)=X_2^2-dX_3^2,\]
for $d\in \ZZ$ a non-square.
Similarly, points where there is a prime
$p>M$ which divides $x_0$ and $x_1$ as well as
$F_1(\x),\dots,F_r(\x)$ contribute 
\[\ll B^{n-1+\ep}M^{-1}+B^{n-2+\ep}.\]
We should note though that in order to assert that
$Q_0(x_2,\dots,x_n)=0$ has $O(B^{n-3+\ep})$ integral solutions in
$[-B,B]^{n-1}$ we need to use the fact that $d$ is not a square. We
then have to estimate $N_1(B,M)$, and we may take
$F_i(\b{X})=K_i(X_1,\dots,X_n)$ to be independent of $X_0$.

As before we change notation, replacing $Q_0(X_2,\dots,X_n)$ by
$R(X_1,\dots,X_m)$ with $m=n-1$, and $R(X_1,\dots,X_m)=X_1^2-dX_2^2$.
However, instead of using ``$R$-good'' primes we will employ a
different classification. We will say that a prime $p$ is ramified if
$p\mid 2d$, and otherwise is split if $d$ is a square modulo $p$, and
inert if $d$ is a non-square modulo $p$. Suppose that $q_1$ is a product of
distinct  split primes, and $q_2$ a product of distinct inert
primes. We define the lattices 
\[\sL(\rho;q_1,q_2)=
\{\x\in \ZZ^m:\,x_1\equiv \rho x_2\mod{q_1},\, x_1\equiv x_2\equiv 0 \mod{q_2}
\},\]
for integers $\rho$ in the set
\[Z(q_1)=\{\rho\mod{q_1}:\,\rho^2\equiv d\mod{q_1}\}.\]
These lattices have $\det(\sL(\rho;q_1,q_2))=q_1q_2^2$ for each
$\rho\in Z(q_1)$. Moreover it is clear that
$\card Z(q_1)=2^{\omega(q_1)}$.  We then have the following result, which
will replace Lemma \ref{LL}.

\begin{lemma}\label{LL+}
  Suppose that $R(X_1,\dots,X_m)=X_1^2-dX_2^2$, where $d\in \ZZ$ is
  a non-square.  Let $q_1$ be a product of
distinct split primes, and $q_2$ a product of distinct inert
primes.  Then
\[\left\{\x\in \ZZ^m: R(\x)\equiv 0\mod{q_1q_2}\right\}\subseteq 
\bigcup_{\rho\in Z(q_1)}\sL(\rho;q_1,q_2).\]
Moreover, for each of the lattices
$\sL(\rho;q_1,q_2)$ the largest successive minimum is $O(q_1^{1/2}q_2)$, with an
implied constant depending only on $d$.
\end{lemma}
\begin{proof}
  For any split prime $p$, and any $x_1,x_2\in\ZZ$ satisfying
  $x_1^2\equiv dx_2^2\mod{p}$, there is an integer $\rho$ for which
  $\rho^2\equiv d\mod{p}$ and $x_1\equiv\rho x_2\mod{p}$. Moreover, for
  any inert prime $p$ we have $x_1\equiv x_2\equiv 0\mod{p}$
whenever
  $x_1^2\equiv dx_2^2\mod{p}$. It follows via
  the Chinese Remainder Theorem that the lattices $\sL(\rho;q_1,q_2)$ with
  $\rho\in Z(q_1)$ cover all solutions of
  $R(\x)\equiv 0\mod{q_1q_2}$. Finally, $\sL(\rho;q_1,q_2)$ has a basis
  consisting of the $m-2$
 unit coordinate vectors $\b{e}_3,\dots,\b{e}_m$, together with two
 further vectors $(q_2\b{a},0,\dots,0)$ and $(q_2\b{b},0,\dots,0)$, where
 $\b{a}$ and $\b{b}$ are 2-dimensional vectors forming a basis for the
 lattice 
\[\sL_0(\rho)=\{\x\in \ZZ^2:\,x_1\equiv \rho x_2\mod{q_1}\}.\]
This lattice has determinant $q_1$, and successive minima satisfying
$\lambda_1\lambda_2\ll q_1$. However, for any non-zero vector
$\x\in\sL_0(\rho)$ one has
\[x_1^2-dx_2^2\equiv \rho^2x_2^2-\rho^2x_2^2\equiv 0\mod{q_1}.\]
Moreover $x_1^2-dx_2^2$ cannot vanish, since $d$ is not a square. We
therefore deduce that
\[q_1\le|x_1^2-dx_2^2|\le |d|\cdot ||\x||_2^2.\] 
Thus we must have $\lambda_1\gg q_1^{1/2}$, and hence
$\lambda_2\ll q_1^{1/2}$.  It follows that the vectors $\a$ and $\b{b}$
above may be chosen both to have length $O(q_1^{1/2})$, so that the
largest successive minimum of $\sL(\rho;q_1,q_2)$ is $O(q_1^{1/2}q_2)$, as
required.
\end{proof}

Now, following the argument in Section \ref{s:quadricL} we take
$q_1=q_1(x_1)$ to be the product of split
primes dividing $x_1$, and similarly $q_2=q_2(x_1)$ to be the product of inert primes
dividing  $x_1$. We write $q=q_1q_2$. As before, we weaken the condition 
$x_1\mid Q_0(x_2,\dots,x_n)$, requiring only that
$q\mid Q_0(x_2,\dots,x_n)$. In analogy to (\ref{N1BM}) 
there exists $a^{(q)}$ such that 
\beql{XX}
  N_1(B,M)\ll 
B^{\ve}\sum_{q=q_1q_2\le B}\sum_{\rho\in Z(q_1)}N(B,M,q,\rho,a^{(q)}),
\eeq
where $N(B,M,q,\rho,a)$ is the number of 
$\x=(x_2,\dots,x_n)\in\sL(\rho;q_1,q_2)$ in the box $|\x|\le B$ for which
the polynomials $K_i(a,x_2,\dots,x_n)$ all have a common prime divisor
$p> M$.

The argument then proceeds as before, but without the need to handle
separately lattices where the largest successive minimum is big.  If
the successive minima of $\sL(\rho;q_1,q_2)$ are
$\lambda_1\le\dots\le\lambda_m$ (with $m=n-1$) we apply Lemma~\ref{t:ekedahl} to vectors $\b{w}$ with $w_i\ll B/\lambda_i$ to show that
\[N(B,M,q,\rho,a^{(q)})\ll
\frac{V\log B}{M\log M}+\frac{V\log B}{B_{\mathrm{min}}},\]
with $V=\prod_{i=1}^m(B/\lambda_i)$. Since
\[\prod_{i=1}^m\lambda_i\geq 
\det(\sL(\rho;q_1,q_2))=q_1q_2^2\]
we find that
\[N(B,M,q,\rho,a^{(q)})\ll
\frac{B^m\log B}{q_1q_2^2M\log M}+\frac{B^{m-1}\lambda_m\log B}{q_1q_2^2}.\]
According to Lemma \ref{LL+} we have
$\lambda_m\ll q_1^{1/2}q_2$. Since $\card Z(q_1)\ll B^{\ep}$ we then
deduce from (\ref{XX}) that
\begin{align*}
  N_1(B,M)&\ll 
B^{2\ve}\sum_{q_1q_2\le B}\left\{\frac{B^m\log B}{q_1q_2^2M\log M}+
\frac{B^{m-1}\log B}{q_1^{1/2}q_2}\right\}\\
&\ll 
B^{3\ve}\left\{\frac{B^m}{M\log M}+B^{m-1/2}\right\}.
\end{align*}
On recalling that $m=n-1$ we see that this is sufficient for
Theorem \ref{MT++}, after re-defining $\ep$.

\section{Proof of Corollary \ref{cor:gcd}: coprime polynomials}\label{s:coprime}

The implied constants in this section are allowed to depend on $Q,f$ and $g$.
Assume that $Q$ is an indefinite quadratic form of rank at least
$5$. For any square-free $q\in \NN$ and any vector
$\a\in(\ZZ/q\ZZ)^n$, we shall require an asymptotic formula for   
\begin{equation}\label{eq:count}
N(B;q,\a)=\#\{\x\in \ZZ^n\cap[-B,B]^n: Q(\x)=0,~ \x\equiv \a\mod{q}\},
\end{equation}
as $B\to \infty$, in which the error term depends explicitly on $q$. 
In fact there exist constants $\delta,\Delta>0$ such that 
\begin{equation}\label{eq:asymptotic}
N(B;q,\a)=c(q,\a) B^{n-2}+O(q^\Delta B^{n-2-\delta}),
\end{equation}
where the implied constant depends on $Q$ but not on $\a$ or $q$.
Assuming that $q$ is square-free and that $Q(\a)\equiv 0\mod{q}$, the
leading constant is positive and  
takes the shape 
$$
c(q,\a)=
\sigma_\infty \prod_{p\nmid q}\sigma_p\prod_{p\mid q} \sigma_p(\a).
$$
Here $\sigma_\infty$ is  the density of real zeros of $Q$, which is
independent of $q$ and $\a$. Moreover
\[\sigma_p=\lim_{k\to \infty}p^{-k(n-1)}\nu(p^k)\;\;\;\mbox{and}\;\;\;
\sigma_p(\a)=\lim_{k\to \infty}p^{-k(n-1)}\nu(p^k;p,\a),\]
with 
\[\nu(p^k)=\#\left\{\x\in (\ZZ/p^k\ZZ)^n: Q(\x)\equiv
0\mod{p^k}\right\}\]
and
\[\nu(p^k;p,\a)=\#\left\{\x\in (\ZZ/p^k\ZZ)^n: 
\begin{array}{l}Q(\x)\equiv 0\mod{p^k}\\ \x\equiv \a\mod{p}\end{array}
\right\},\]
for every  prime $p$. As part of the circle method analysis one shows
that all the limits involved  exist. We shall write 
$c=c(1,\mathbf{0})$
for brevity. The proof of \eqref{eq:asymptotic}  is a standard
application of the Hardy--Littlewood circle method and will not be
repeated here. (A more refined treatment of the analogous smoothly
weighted  
counting function is found in \cite[Thm.~4.1]{BL}, in which any
values $\Delta>n/2$ and $\delta<n/2-2$ are shown to be admissible.)

We remark that the analogous statement for quadratic forms of rank 4
is false in general, even for the forms $X_0X_1-(X_2^2-dX_3^2)$ with
non-square $d$ that are considered in Theorem \ref{MT++}.
We refer the reader to Linqvist \cite{Lind} for
further details on this phenomenon.

Let $M>\xi>1$ and let $P_\xi=\prod_{p\leq \xi}p.$ We shall tackle
Corollary \ref{cor:gcd} by observing that  
\begin{equation}\label{eq:123}
  \#S_1-\#S_2-\#S_3\leq
  \#\{\x\in \mathcal{R}_{f,g}\cap[-B,B]^n: Q(\x)=0\}\leq \#S_1,
\end{equation}
where 
\begin{align*}
S_1=\{\x\in \ZZ^n\cap[-B,B]^n: Q(\x)=0, ~\gcd(f(\x),g(\x),P_\xi)=1\},
\end{align*}
$S_2$ is the set of $\x\in S_1$ for which $p\mid \gcd(f(\x),g(\x))$
for some $p\in (\xi,M]$,  and finally  $S_3$ is the set of
$\x\in S_1$ for which $p\mid \gcd(f(\x),g(\x))$ for some $p>M$.
  Noting that $f=g=0$ cuts out a codimension $2$ subvariety in the
  hypersurface $Q=0$, it follows from  Theorem \ref{MT+} that 
\begin{equation}\label{eq:S3}
\#S_3\ll B^{n-2+\ve} M^{-1} +B^{n-2-1/(n-1)+\ve},
\end{equation}
for any $\ve>0$.

Turning to the size of $S_1$ we use inclusion--exclusion to deduce that 
\begin{align*}
\#S_1 &=\sum_{q\mid P_\xi} \mu(q) \#\left\{\x\in \ZZ^n\cap[-B,B]^n: 
\begin{array}{l}
Q(\x)=0\\ f(\x)\equiv g(\x)\equiv 0 \mod{q}\end{array}\right\}\\
&=\sum_{q\mid P_\xi} \mu(q) \sum_{\substack{\a \in (\ZZ/q\ZZ)^n\\ 
Q(\a)\equiv f(\a)\equiv g(\a)\equiv 0 \mod{q}}} N(B;q,\a).
\end{align*}	
Note that there are at most $q^n$ vectors $\a$ which contribute to the
final sum. Invoking \eqref{eq:asymptotic}, and recalling
that $c=c(1,\mathbf{0})$, it follows that 
\beql{S1}
\#S_1=c B^{n-2}\sum_{q\mid P_\xi} \mu(q)g(q) +O\left(B^{n-2-\delta} 
\sum_{q\mid P_\xi} q^{n+\Delta}\right),
\eeq
with
\[g(q)=\sum_{\substack{\a \in (\ZZ/q\ZZ)^n\\ Q(\a)\equiv
    f(\a)\equiv g(\a)\equiv 0 \mod{q}}}\hspace{-0.2cm}
\prod_{p\mid q}\lim_{k\to \infty}\frac{\nu(p^k;p,\a)}{\nu(p^k)}.\]
The error term here is found to be 
\[\ll B^{n-2-\delta} \prod_{p\leq \xi} p^{n+\Delta}= B^{n-2-\delta} 
\exp\left((n+\Delta)\sum_{p\leq \xi} \log p\right)
\leq B^{n-2-\delta}  e^{2(n+\Delta)\xi},\]
if $\xi\gg 1$, by the prime number theorem.

For the main term in (\ref{S1}) we wish to extend the product to run
over all primes. The function $g(q)$
is multiplicative and for any prime $p$ we have 
\[g(p)=\lim_{k\to \infty}\frac{\nu_0(p^k)}{\nu(p^k)},\]
where
\[\nu_0(p^k)=\#\left\{\x\in (\ZZ/p^k\ZZ)^n: 
\begin{array}{l}Q(\x)\equiv 0\mod{p^k}\\
f(\x)\equiv g(\x)\equiv 0 \mod{p}\end{array}\right\}.\]
It is clear that $g(p)\le 1$ for every prime, but we will need a
better bound for large $p$. Suppose that $Q$ has rank $r\ge 5$. If $p$
is odd, we may diagonalize $Q$ modulo $p^k$ as
$\mathrm{Diag}(d_1,\dots,d_r,0,\dots,0)$ with respect to a
suitable basis, and if $p$ is large enough
we will have $p\nmid d_i$ for $1\le i\le r$. Using this new basis we
see that $\nu(p^k;p,\a)$
counts $\x\in (\ZZ/p^k\ZZ)^n$ with $\x\equiv\a\mod{p}$ and
\beql{de}
\sum_{i=1}^rd_ix_i^2\equiv 0\mod{p^k}.
  \eeq
If we write $\mathbf{b}=(a_1,\dots,a_r)$ it follows that
$\nu(p^k;p,\a)=p^{(k-1)(n-r)}\xi(p^k;p,\mathbf{b})$, where
$\xi(p^k;p,\mathbf{b})$
counts $\x\in (\ZZ/p^k\ZZ)^r$ with $\x\equiv\mathbf{b}\mod{p}$, such
that (\ref{de}) holds. When $\mathbf{b}\not\equiv\mathbf{0}\mod{p}$ we
find that $\xi(p^k;p,\mathbf{b})=p^{(k-1)(r-1)}$, by Hensel's Lemma,
so that $\nu(p^k;p,\a)=p^{(k-1)(n-1)}$. For large $p$ the number of
$\a\mod{p}$ for which $Q(\a)\equiv f(\a)\equiv g(\a)\equiv 0\mod{p}$
will be $O(p^{n-3})$, so that vectors $\a$ for which 
$\mathbf{b}\not\equiv\mathbf{0}\mod{p}$ contribute $O(p^{k(n-1)-2})$
to $\nu_0(p^k)$. On the other hand, a standard calculation gives
$\xi(p^k;p,\mathbf{0})\ll p^{r+(k-2)(r-1)}$, so that
$\nu(p^k;p,\a)\ll p^{(k-1)(n-1)+1}$ for those $\a$ for which
$\mathbf{b}\equiv\mathbf{0}\mod{p}$. The number of such $\a$ is
$p^{n-r}\le p^{n-5}$, whence this case contributes $O(p^{k(n-1)-3})$
to $\nu_0(p^k)$. However a standard analysis shows that
$\nu(p^k)\gg p^{k(n-1)}$, so that
\[g(p)\ll\lim_{k\to\infty}\frac{p^{k(n-1)-2}+p^{k(n-1)-3}}{p^{k(n-1)}}\ll p^{-2}.\]
 Since $g(q)$ is multiplicative we then have $g(q)=O(q^{-3/2})$ for
 any square-free $q\in \NN$.  Hence it follows that 
\[\sum_{q\mid P_\xi} \mu(q)g(q)-\sum_{q=1}^{\infty}\mu(q)g(q) 
\ll  \sum_{q>\xi} \frac{1}{q^{3/2}}\ll \xi^{-1/2}.\] 
Our work so far has therefore shown that 
\begin{equation}\label{eq:S1}
\#S_1=c B^{n-2} \prod_{p} \mu_{Q,p}(\mathcal{R}_{f,g}) 
+O(\xi^{-1/2}B^{n-2})+O(e^{2(n+\Delta)\xi} B^{n-2-\delta}) ,
\end{equation}
where 
$\mu_{Q,p}(\mathcal{R}_{f,g})$
is as in the statement of Corollary \ref{cor:gcd}.

To handle $S_2$, we
note that 
$$
\#S_2
\leq \sum_{p\in (\xi,M]  }
\sum_{\substack{\a \in (\ZZ/p\ZZ)^n\\ 
Q(\a)\equiv f(\a)\equiv g(\a)\equiv 0 \mod{p}}} N(B;p,\a).
$$
But 
  \eqref{eq:asymptotic} allows us to conclude that 
\begin{equation}\label{eq:coro}
\sum_{\substack{\a \in (\ZZ/p\ZZ)^n\\ 
Q(\a)\equiv f(\a)\equiv g(\a)\equiv 0 \mod{p}}} N(B;p,\a)
\ll  \frac{B^{n-2}}{p^2}+ p^{n+\Delta}B^{n-2-\delta} ,
\end{equation}
since $g(p)=O( p^{-2})$.  
Summing over $p\in (\xi,M] $ it  follows that
\begin{equation}\label{eq:S2}
\#S_2 \ll   \frac{B^{n-2}}{\xi\log \xi}  +
M^{n+1+\Delta}B^{n-2-\delta}.
\end{equation}

We now return to \eqref{eq:123}, taking  $\xi=\sqrt{\log B}$ and 
$M=B^{\delta/(2(n+1+\Delta))}$. Making the choice
$\ve=\delta/(4(n+1+\Delta))$ in \eqref{eq:S3}, and combining it with
\eqref{eq:S1} and \eqref{eq:S2}, it follows that  
$$
\#\{\x\in \mathcal{R}_{f,g}\cap[-B,B]^n: Q(\x)=0\}=
c  B^{n-2}  \prod_{p} \mu_{Q,p}(\mathcal{R}_{f,g}) 
 +O\left(\frac{B^{n-2}}{(\log B)^{1/4}}\right).
$$
Finally, we divide both sides by $N(B;1,\mathbf{0})$ and reapply
\eqref{eq:asymptotic}, before taking a limit $B\to \infty$ in order to
complete the proof of Corollary \ref{cor:gcd}.

\section{Proof of Corollary \ref{c:SA}: arithmetic purity}\label{s:SA}

Let $m\geq 4$ and let  $X\subset \PP^{m}$ be a smooth hypersurface defined 
by a non-singular indefinite quadratic form $Q\in \ZZ[X_0,\dots,X_m]$.
Let  $Z\subset X$ be a codimension 2 subvariety and put  $U=X\setminus Z$.
To establish  strong approximation off $\infty$ on $U$ we must show  that 
for any point $(P_p)_p$ in the set of finite adelic points $U(\A_\QQ
^f)$ and for any finite set $S$ of primes, there exists a point 
$P\in U(\QQ)$ which is arbitrarily close to $P_p$ for all $p\in S$.

There exists  an integral  model $\mathcal{U}$ for $U$ 
over $\ZZ$. It will suffice to show that there exists a point 
$P\in U(\QQ)$ with $P\in \mathcal{U}(\ZZ_{p})$ for all $p\not\in S$,
such that $P$ is arbitrarily close to $P_p$ for all $p\in S$. 

Let  $\mathcal{Z}$ be the scheme-theoretic closure of $Z$ in $\PP_\ZZ^m$.
We may suppose that $\mathcal{Z}$ is cut out by equations
$$
F_1(X_0,\dots,X_m)=\dots=F_r(X_0,\dots,X_m)=0,
$$
for  $F_1,\dots,  F_r\in \ZZ[X_0,\dots,X_m]$ such that the
intersection with $Q=0$ has codimension $3$ in $\PP^{m}$. 
For any prime $p$, elements of $\mathcal{U}(\ZZ_p)$  correspond to
vectors $\x\in \ZZ_p^{m+1}$ for which $Q(\x)=0$ and  
$$
\min\{\val_p (F_1(\x)), \dots, \val_p (F_r(\x))\}=0.
$$
Let $C\in \ZZ$ be a product of primes in $S$,  chosen so that 
$P_p'=CP_p\in \ZZ_p^{m+1}$ for all $p\in S$.
By the Chinese Remainder Theorem we can find a vector $\a\in\ZZ^{m+1}$ which is 
arbitrarily close to $P_p'$ for all  $p\in S$. A vector $\x\in
\ZZ^{m+1}$ representing a point in  
 $ U(\QQ)$ is then  close to $\a$ in the $p$-adic topology for all
$p\in S$ if any only if  
$\x\equiv \a \mod M$,
for a suitable positive integer $M$ built from the primes in $S$.
In order to establish Corollary~\ref{c:SA}, it will suffice to prove
the existence of a vector $\x\in \ZZ^{m+1}$, satisfying $Q(\x)=0$ and 
 $$
 \gcd(p,F_1(\x),\dots,F_{r}(\x))=1 \text{ for all $p\not \in S$},
 $$
 and  for which    $\x\equiv \a \mod M$.
Indeed, once this is achieved the vector
$C^{-1}\x$ will represent a point $P\in U(\QQ)$ which is $p$-adically
close to $P_p$ for all  
 $p\in S$ and which belongs to $\mathcal{U}(\ZZ_p)$ for all $p\not\in S$.

Finally, to deduce the existence of the vector $\x$ we count the
number of such vectors in the box  $[-B,B]^{m+1}$, 
as $B\to \infty$. But then we are  once more in the situation
considered in Section \ref{s:coprime}, where we dealt with exactly
this  question when  $S=\emptyset$ and $r=2$. Extending the argument to general
$S$ and $r$ is routine and will not be repeated here.

\section{Proof of Corollary \ref{c:codim}: local solubility}\label{s:Ekedahl}

The aim of this section is to prove Corollary  \ref{c:codim}, the main  tool for which 
is Theorem \ref{MT+}.  The strategy for our argument closely follows  the proof of 
Lemma 20 in work of Poonen and Stoll
\cite{PS}, as further developed  by Bright, Browning and Loughran \cite[Section 3]{wa}.
We shall write $m=n-1$ in order to simplify notation.  
Let $X\subset \PP^{n-1}$ be a hypersurface defined by an indefinite  quadratic form 
$Q\in \ZZ[X_1,\dots,X_n]$
of rank at least $5$. 
Let $\pi:Y \to X$ be a  morphism as in the statement of the theorem.
Thus the fibre of $\pi$ over every point of 
	codimension 1 is split and the generic fibre of $\pi$ is geometrically
	integral. Appealing to  Corollary~3.7 of \cite{wa}, it then follows
that there exist a finite set  $S$ of places of $\QQ$, together
	with models $\mathcal{Y}$ and $\mathcal{X}$ of $Y$ and $X$ over $\Spec( \ZZ_S)$
	and a  closed subset $\mathcal{Z} \subset \mathcal{X}$ of codimension 
	at least 2, such that 
	the map
	\[
	(\mathcal{Y} \setminus \pi^{-1}(\mathcal{Z})) (\ZZ_p) \to 
	(\mathcal{X} \setminus \mathcal{Z})(\ZZ_p)
	\]
	is surjective 
	for all primes  $p \not\in S$. We may assume without loss of generality that $S$ contains the infinite place. 
It follows that 
\begin{equation}\label{eq:include}
\{x \in X(\ZZ_p): x \bmod p \not \in \mathscr{Z}(\FF_p) \}
\subset \pi(Y(\QQ_p)),
\end{equation}
for all sufficiently large primes $p$.
We  proceed under the assumption that $\mathcal{Z}$ is cut out from
$\mathcal{X}$ by a system of forms $F_1,\dots,F_r\in
\ZZ[X_1,\dots,X_n]$. 
We henceforth allow all of the implied constants in this section to
depend on $F_1,\dots,F_r$ and on $X$. 

For any field $k$ and any 
subset $\Omega\subset \PP^{n-1}(k)$, we shall denote by 
$\Omega^{\text{aff}}$ the affine cone of $\Omega$.
For each prime $p$ we let $\Omega_p =\pi(Y(\QQ_p))^{\text{aff}}\cap \ZZ_p^n$. At the infinite place we put 
$\Omega_\infty=\{\x\in \pi(Y(\RR))^{\text{aff}}: |\x|\leq 1\}\cap \RR^n$.
Let $\mu_\infty$ and $\mu_p$ be the Haar measures on $\RR^n$ and $\ZZ_p^n$, respectively.
It follows from Lemma 3.9 of \cite{wa} that $\Omega_\nu$ is measurable with respect to $\mu_\nu$, with 
$
\mu_\nu(\partial \Omega_\nu )=0$ and $\mu_\nu(\Omega_\nu)>0.$
The proof of this result is based on the Tarski--Seidenberg--Macintyre theorem, as  applied here to the affine cone of the map obtained by composing $\pi$ with the $\QQ$-birational map to $\PP^{n-2}$ 
admitted by  $X$. If $x=(x_1:\dots:x_n)$ denotes the projective point in $\PP^{n-1}$ associated to a vector 
$\x=(x_1,\dots,x_n)$, then we have 
$$
\mu_p(\Omega_p)=\lim_{k\to \infty} \frac{ \#\left\{
\x\in (\ZZ/p^k\ZZ)^n: Q(\x)\equiv 0\mod{p^k}, ~\pi^{-1}(x)(\QQ_p)\neq \emptyset
\right\}}{p^{k(n-1)}}
$$
and 
$$
\mu_\infty(\Omega_\infty)= \lim_{\delta\to 0}\frac{1}{2\delta} \meas\left\{ 
\x\in [-1,1]^n: |Q(\x)|<\delta,~\pi^{-1}(x)(\RR)\neq \emptyset\right\}.
$$

Recall the notation 
$\mu_Q(\mathcal{S})$ that was introduced in 
\eqref{eq:muQ}, for any subset $\mathcal{S}\subset \ZZ^n$.
In order to prove Theorem \ref{c:codim}, it will   suffice to study 
$$
\mu_Q(\mathcal{R}_{\text{loc}})=
\lim_{B\to \infty} \frac{\#\{\x\in
  \mathcal{R}_\text{loc}\cap[-B,B]^n: Q(\x)=0\}}{\#\{\x\in \ZZ^n\cap[-B,B]^n:
  Q(\x)=0\}}, 
$$
where
$$
\mathcal{R}_{\text{loc}}=\left\{\x\in \ZZ^n:  \x \in \Omega_\nu \text{ for all  places $\nu$}
\right\}.
$$ 
Suppose first that there exists 
 $M\in \NN$ such that $\Omega_p=\ZZ_p^n$ for all primes $p>M$. 
 We let  $P=\prod_{p\leq M}\Omega_p$ and $Q=\prod_{p\leq M}(\ZZ_p^n\setminus \Omega_p)$.
The sets  $\Omega_p$ and $\ZZ_p^n\setminus \Omega_p$ have boundary of measure zero.
Hence by compactness we can cover the closure $\overline{P}$ of $P$ 
by a finite number of
boxes $\prod_{p\leq M} I_p$, 
the sum of whose measures is 
 arbitrarily close to the measure 
 $\prod_p \mu_p(\Omega_p)$
 of $\overline{P}$, 
where each 
$I_p\subset \ZZ_p^n$ is a 
 cartesian product of closed balls of the shape $\{x\in \ZZ_p: |x-a|_p\leq b\}$, for $a\in \ZZ_p$ and $b\in \RR$.
Similarly, the closure $\overline{Q}$ of $Q$ is covered by a finite
number of boxes  
$\prod_{p\leq M} J_p$, say, the sum of whose measures approximates the 
measure $1-\prod_p \mu_p(\Omega_p)$ 
of $\overline{Q}$ to arbitrary precision.

It follows from the Chinese Remainder Theorem 
that  there exist a vector $\a_M\in \ZZ^n$ and a modulus $q_M\in \NN$, depending on $M$,  such that for any $\x\in \ZZ^n$ we have 
$\x\in \prod_{p\leq M} I_p$ if and only if $\x\equiv \a_M \mod{q_M}$.
Let 
$$
N_\mathfrak{R}(B;q,\a) = \#\left\{\x\in \ZZ^n\cap B\mathfrak{R}: Q(\x)=0,~ \x\equiv \a\mod{q}\right\},
$$
for any $\mathfrak{R}\subset \RR^n$ of finite measure, any $q\in \NN$ and any $\a\in (\ZZ/q\ZZ)^n$.
When $\mathfrak{R}=[-1,1]^n$  we simply write $N(B;q,\a)$ and thereby recover the counting function that was introduced in \eqref{eq:count}.
It now follows from \eqref{eq:asymptotic} that 
\begin{align*} 
\mu_Q\left(\prod_{p\leq M} I_p\right)
&=\lim_{B\to \infty} \frac{N_{\Omega_\infty}(B;q_M,\a_M)}{N(B;1,\0)}\\
&=\frac{ \mu_\infty(\Omega_\infty)}
{\sigma_\infty}
\prod_{p}
	\frac{\mu_p(I_p)}{\sigma_p}.
	\end{align*}
Similarly, 
\begin{align*} 
\mu_Q\left(\prod_{p\leq M} J_p\right)
&=\frac{ \mu_\infty(\Omega_\infty)}{\sigma_\infty}
\prod_{p}
	\frac{\mu_p(J_p)}{\sigma_p}.
	\end{align*}
Combining  these facts, we  are therefore done when 
 there exists $M$ such that $\Omega_p=\ZZ_p^n$ for all primes $p>M$.  
 
We now turn to the general case. For $M \leq M'\leq \infty$ 
and $B > 0$, let
$$
f_{M,M'}(B) = \frac{\#\{\x \in \ZZ^n\cap B\Omega_\infty:  Q(\x)=0, ~\x \in \Omega_p \text{ for all }p \in [M,M')\}}
{N(B;1,\0)}.
$$
Put  $f_{M}(B)=f_{1,M}(B)$ and 
note that this is a non-increasing function of $M$.
According to  \eqref{eq:include}, there are forms 
$F_1,\dots, F_r\in \ZZ[X_1,\dots,X_n]$ whose common zero locus meets 
$X$ in a codimension $3$ subset of 
$\PP^{n-1}$, for which 
\begin{align*}
f_M(B)-f_\infty(B)
&=
 \frac{\#\{\x \in \ZZ^n\cap B\Omega_\infty: Q(\x)=0, ~ \text{$\exists$ $p>M$,  }\x \not\in \Omega_p \}}
{N(B;1,\0)}\\
&\leq 
 \frac{E(B,M)}
{N(B;1,\0)},
\end{align*}
where
$E(B,M)$ is  the number of
$\x \in \ZZ^n$ such that $ Q(\x)=0$ and $|\x|\leq B$, and for which 
$F_1(\x),\dots,F_r(\x)$ have a common prime divisor $p>M$.
We have ${N(B;1,\0)}\gg B^{n-2}$ by \eqref{eq:asymptotic}.
We may sort $E(B,M)$  into two contributions. 
Let $\eta>0$ be a parameter at our disposal. 
The contribution from $\x$ for which 
$F_1(\x),\dots,F_r(\x)$ have a common prime divisor $p>B^{\eta}$ 
is seen to be 
$$
\ll B^{\ve-\eta} + \frac{B^\ve}{B^{1/(n-1)}}, 
$$
by   Theorem \ref{MT+}.
Next, the contribution from $\x$ for which 
$F_1(\x),\dots,F_r(\x)$ have a common prime divisor $p\in(M,B^{\eta}]$ 
is at most 
\begin{align*}
\sum_{M<p\leq B^{\eta}}
\hspace{-0.3cm}
\sum_{\substack{
\a \in (\ZZ/p\ZZ)^n\\
Q(\a)\equiv F_1(\a)\equiv \dots\equiv  F_r(\a)\mod{p}}}
\hspace{-0.6cm}
N(B;p,\a)
&\ll \frac{B^{n-2}}{M\log M}
+B^{n-2-\delta+\eta(n+1+\Delta)}, 
\end{align*}
by \eqref{eq:coro}.
It  follows that 
$$
f_M(B)-f_{\infty}(B)\ll B^{\ve-\eta} + \frac{B^\ve}{B^{1/(n-1)}} 
+ \frac{1}{M\log M}+\frac{B^{\eta(n+1+\Delta)}}{B^\delta}, 
$$
for any $\ve>0$. On taking
$\eta=\delta/(n+2+\Delta)$ and choosing $\ve$ sufficiently small, 
we obtain
\begin{equation}	\label{eqn:uniform}
	\lim_{M \to \infty} \limsup_{B \to \infty} (f_M(B) - f_{\infty}(B)) = 0.
\end{equation}
Moreover, our work so far shows that
\begin{equation} \label{eqn:pointwise}
	\lim_{B \to \infty} f_{M,M'}(B) =
	\frac{ \mu_\infty(\Omega_\infty)}{\sigma_\infty}
\prod_{M\leq p<M'}
	\frac{\mu_p(\Omega_p)}{\mu_p\left(X(\QQ_p)^{\text{aff}}\cap \ZZ_p^n\right)},
\end{equation}
for all $M<M'<\infty$.
Combining  \eqref{eqn:uniform} and \eqref{eqn:pointwise}, we conclude that 
\begin{align*}
\lim_{B \to \infty} f_\infty(B)
&= \lim_{M \to \infty} \lim_{B \to \infty} f_M(B) \\
&=
	\frac{ \mu_\infty(\Omega_\infty)}{\sigma_\infty}
\lim_{M\to \infty}
\prod_{p<M}
	\frac{\mu_p(\Omega_p)}{\sigma_p}
	\end{align*}
To complete the proof, it suffices to show the convergence of the above infinite product.
In order to  apply  Cauchy's criterion we need to check that
$$
\lim_{M\to\infty}\sup_{M'\in\NN}\left|1 -\prod_{M \leq p<M+M'} 
	\frac{\mu_p(\Omega_p)}{\sigma_p}
	\right| =0.
$$
But \eqref{eqn:pointwise}  implies that the left hand side is
$$
\frac{\sigma_\infty}{\mu_\infty(\Omega_\infty)}\lim_{M \to \infty}
\sup_{M'\in\NN}\lim_{B\to\infty}\left|f_{1}(B)-f_{M,M+M'}(B)\right|, 
$$
which vanishes by a further application of \eqref{eqn:uniform}.
Combining our argument, we have therefore shown 
that 
$$
\mu_Q(\mathcal{R}_\text{loc})=
	\frac{ \mu_\infty(\Omega_\infty)}{\sigma_\infty}
\prod_{p}
	\frac{\mu_p(\Omega_p)}{\sigma_p},
$$
which thereby  completes the proof of Corollary \ref{c:codim}.

\end{document}